\documentclass[11pt]{amsart}
\usepackage{amsmath,amssymb,xypic,array}
\usepackage[T1]{fontenc}
\usepackage{amsfonts}
\usepackage{amsmath}
\usepackage{amssymb}
\usepackage{amsthm}
\usepackage[cp850]{inputenc}
\usepackage{hyperref}
\usepackage{graphicx}

\usepackage{tikz}
\usepackage{longtable}
\usepackage{geometry}
\usepackage{textcomp}
\usetikzlibrary{trees}

\providecommand{\U}[1]{\protect\rule{.1in}{.1in}}
\providecommand{\U}[1]{\protect\rule{.1in}{.1in}}
\providecommand{\U}[1]{\protect\rule{.1in}{.1in}}
\providecommand{\U}[1]{\protect\rule{.1in}{.1in}}
\providecommand{\U}[1]{\protect\rule{.1in}{.1in}}
\input{xy}
\xyoption{all}
\theoremstyle{theorem}

\newtheorem{Theorem}{Theorem}[section]
\newtheorem*{theoremn}{Theorem}

\newtheorem{Lemma}[Theorem]{Lemma}
\newtheorem{Proposition}[Theorem]{Proposition}

\theoremstyle{definition}

\newtheorem*{scr}{Script}
\newtheorem{Definition}[Theorem]{Definition}
\newtheorem{Remark}[Theorem]{Remark}

\newtheorem{Example}[Theorem]{Example}
\newtheorem{Construction}[Theorem]{Construction}
\newtheorem{Convention}[Theorem]{Convention}

\numberwithin{equation}{section}

\newcommand{\arXiv}[1]{\href{http://arxiv.org/abs/#1}{arXiv:#1}}
\def\bibaut#1{{\sc #1}}

\DeclareMathOperator{\rank}{rank}

\DeclareMathOperator{\Hilb}{Hilb}

\DeclareMathOperator{\Proj}{Proj}

\DeclareMathOperator{\Sing}{Sing}

\DeclareMathOperator{\Sec}{Sec}
\DeclareMathOperator{\codim}{codim}

\newcommand{\QED}{\ifhmode\unskip\nobreak\fi\quad {\rm Q.E.D.}} 

\newcommand\bin[2]{{#1\choose #2}}

\newcommand\iso{\cong}

\newcommand{\f}{\varphi}

\newcommand{\G}{\mathbb{G}}

\DeclareMathOperator{\Grass}{Grass}

\renewcommand{\P}{\mathbb{P}}

\newcommand{\rat}{\dasharrow}
\renewcommand{\sec}{\mathbb{S}ec}

\begin{document}
\title{Generalized varieties of sums of powers}

\author[Alex Massarenti]{Alex Massarenti}
\address{\sc Alex Massarenti\\
IMPA\\
Estrada Dona Castorina 110\\
22460-320 Rio de Janeiro\\ Brasil}
\email{massaren@impa.br}

\date{\today}

\maketitle

\begin{abstract}
Let $X\subset\mathbb{P}^{N}$ be an irreducible, non-degenerate variety. The generalized variety of sums of powers $VSP_H^X(h)$ of $X$ is the closure in the Hilbert scheme $\Hilb_{h}(X)$ of the locus parametrizing collections of points $\{x_{1},...,x_{h}\}$ such that the $(h-1)$-plane $\left\langle x_{1},...,x_{h}\right\rangle$ passes trough a fixed general point $p\in\mathbb{P}^{N}$. When $X = V_{d}^{n}$ is a Veronese variety we recover the classical variety of sums of powers $VSP(F,h)$ parametrizing additive decompositions of a homogeneous polynomial as powers of linear forms. In this paper we study the birational behavior of $VSP_H^X(h)$. In particular we will show how some birational properties, such as rationality, unirationality and rational connectedness, of $VSP_H^X(h)$ are inherited from the birational geometry of variety $X$ itself.
\end{abstract}

\tableofcontents

\section*{Introduction}
Let $X\subset\mathbb{P}^{N}$ be an irreducible, non-degenerate variety. The Zariski closure of the union of the linear spaces spanned by collections of $h$ points on $X$ is the $h$-secant variety $\sec_{h}(X)$ of $X$. Secant varieties are central objects in both classical algebraic geometry \cite{CC}, \cite{Za}, and applied mathematics \cite{LM}, \cite{LO}, \cite{MR}.\\
The abstract $h$-secant variety $\Sec_{h}(X)$ is the Zariski closure of the set of couples $(p,H)$ where $p\in\mathbb{P}^{N}$ and $H$ is an $(h-1)$-plane $h$-secant to $X$. In this paper we study the general fiber of the natural map $\pi_{h}:\Sec_{h}(X)\rightarrow\mathbb{P}^{N}$. If $p\in\mathbb{P}^{N}$ is a general point such fiber parametrizes $(h-1)$-planes $h$-secant to $X$ passing through $p$.\\
Let $V_{d}^{n}$ be the Veronese variety obtained as the image of the embedding $\nu_{d}^{n}:\mathbb{P}^{n}\rightarrow\mathbb{P}^{N}$ induced by $\mathcal{O}_{\mathbb{P}^{n}}(d)$. When $X = V_{d}^{n}$ the general fiber of $\pi_{h}:\Sec_{h}(V_{d}^{n})\rightarrow\mathbb{P}^{N}$ parametrizes decompositions of a general homogeneous polynomial $F\in k[x_{0},...,x_{n}]$ as sums of powers of linear forms. In this case the fiber $\pi_{h}^{-1}(F)$ is called the \textit{variety of sums of powers} of $F$ and denoted by $VSP(F,h)$. The varieties $VSP(F,h)$ have been widely studied from both the biregular \cite{IR}, \cite{Mu1}, \cite{Mu2}, \cite{RS} and the birational viewpoint \cite{MMe}. The interest in these varieties increased after \textit{S. Mukai} described the Fano $3$-fold $V_{22}$ and a polarized $K3$ surface of genus $20$ as the varieties of sums of powers of general polynomials $F\in k[x_0,x_1,x_2]_{4}$ and $F\in k[x_0,x_1,x_2]_{6}$ respectively \cite{Mu1}, \cite{Mu2}. Then many authors generalized Mukai's techniques to other polynomials \cite{IR}, \cite{DK}, \cite{RS}, \cite{TZ}. See \cite{Do} for a survey. Recently the variety of sums of powers of a polynomial $F\in k[x_0,...,x_5]_{4}$ has been used to construct particular divisors in the moduli space of cubic $4$-folds \cite{RV}.\\
In this paper, in analogy with the classical case, we denote by $VSP_{G}^{X}(h)$ the general fiber of the map $\pi_{h}:\Sec_{h}(X)\rightarrow\mathbb{P}^{N}$ and we call the varieties $VSP_{G}^{X}(h)$ \textit{generalized varieties of sums of powers}, see Definition \ref{vspg}. The letter $G$ reminds us that we are looking at $VSP_{G}^{X}(h)$ in the Grassmannian $\G(h-1,N)$. In Definition \ref{vsph} we introduce the varieties $VSP_H^X(h)$ as a subvariety of the Hilbert scheme of points $\Hilb_{h}(X)$. However, under a suitable numerical hypothesis $VSP_{G}^{X}(h)$ and $VSP_H^X(h)$ turn out to be birational, see Remark \ref{comp}.\\
Our aim is to investigate the birational behavior of $VSP_H^X(h)$. More precisely we will show how some birational properties of $VSP_H^X(h)$ are inherited from the birational geometry of $X$ itself.\\
In Section \ref{smd} we consider the case when $X$ is a variety of minimal degree. That is an irreducible, non-degenerate variety $X\subset\mathbb{P}^{N}$ such that $\deg(X) = \codim(X)+1$. In this context our main result is Theorem \ref{mindeg}.
\begin{theoremn}
Let $X\subset\mathbb{P}^{N}$ be a variety of minimal degree $\deg(X)=d$. Then $VSP_{H}^{X}(h)$ is irreducible for any $h\geq d$. Furthermore $VSP_{H}^{X}(h)$ is rational if $h = d$, and unirational for any $h\geq d$.
\end{theoremn}
In Section \ref{ldrc} we consider the case when the rational map
$$
\begin{array}{cccc}
\chi: & VSP_{H}^{X}(h) & \dasharrow & \G(h-2,N-1)\\
 & \{x_1,...,x_h\} & \longmapsto & \left\langle x_1,...,x_h\right\rangle
\end{array}
$$
is dominant. Here the Grassmannian $\G(h-2,N-1)$ parametrizes $(h-1)$-planes passing through a general point $p\in\mathbb{P}^{N}$. Then, by studying the general fiber of $\chi$, we get Theorem \ref{rcdeg}.
\begin{theoremn}
Let $X\subset\mathbb{P}^{N}$ be an irreducible variety. Assume $h > N-\dim(X)+1$ and that the general $(h-1)$-dimensional linear section of $X$ is rationally connected. Then the irreducible components of $VSP_{H}^{X}(h)$ are rationally connected.
\end{theoremn}
In Proposition \ref{ver}, with a similar argument, we prove that the irreducible components of the classical varieties of sums of power $VSP(F,h)$ of a general homogeneous polynomial $F\in k[x_0,...,x_n]_{d}$ are rationally connected as soon as $h\geq \frac{d(N+1)-n}{d}$.\\
Furthermore in Theorem \ref{RC2} we obtain a generalization of \cite[Theorem 4.1]{MMe} by considering an arbitrary unirational variety instead of the Veronese variety.
\begin{theoremn}
Let $X\subset\mathbb{P}^{N}$ be a unirational variety. Assume that for some positive integer $k<n$ the number $\overline{h} = \frac{N}{k+1}$ is an integer and 
$$\frac{N+n+2}{n+1} \leq\overline{h} < N-n+1.$$ 
Then the irreducible components of $VSP_{H}^{X}(h)$ are rationally connected for $h\geq \overline{h}$.
\end{theoremn}
Finally in Section \ref{candec} we consider the cases when there exists a canonical decomposition. This means that there exists a positive integer $\overline{h}$ such that $\sec_{\overline{h}}(X) = \mathbb{P}^{N}$ and $VSP_{H}^{X}(\overline{h})$ is a single point, that is through a general point $p\in\sec_{\overline{h}}(X)$ passes exactly one $(\overline{h}-1)$-plane $\overline{h}$-secant to $X$. Our main result is Theorem \ref{uu}.
\begin{theoremn}
Let $X\subset\mathbb{P}^{N}$ be an irreducible rational variety. Assume that there exists a positive integer $\overline{h}$ such that $\sec_{\overline{h}}(X) = \mathbb{P}^{N}$ and $VSP_{H}^{X}(\overline{h})$ is a single point. Then $VSP_{H}^{X}(h)$ is unirational for any $h\geq\overline{h}$.
\end{theoremn}
Furthermore we study the uniqueness of the decomposition when $X = V_{d}^{n}$ is a Veronese variety. The first results in this direction are due to \textit{J. J. Sylvester} \cite{Sy}, \textit{D. Hilbert} \cite{Hi}, \textit{H. W. Richmond} \cite{Ri}, and \textit{F. Palatini} \cite{Pa}. In the last few years this problem has been studied in \cite{Me1} and \cite{Me2}. As widely expected, the canonical decomposition very seldom exists \cite[Theorem 1]{Me2}. However it is known that a general homogeneous polynomial $F\in k[x_{0},...,x_{n}]_{d}$ admits a canonical decomposition as a sum of $d$-th powers of linear forms in the following cases.
\begin{itemize}
\item[-] $n = 1$, $d = 2h-1$ \cite{Sy},
\item[-] $n = d = 3$, $h = 5$ \cite{Sy},
\item[-] $n = 2$, $d = 5$, $h = 7$ \cite{Hi}.
\end{itemize}
In Proposition \ref{sy1} and Theorems \ref{hi}, \ref{sy} we give very simple and geometrical proofs of these three facts. Furthermore in each one of the listed cases we give an algorithm to reconstruct the decomposition of a given polynomial.

\section{Notation and Preliminaries}
We work over an algebraically closed field $k$ of characteristic zero. 
\subsubsection*{Varieties of sums of powers and secant varieties}
Let $X\subset\P^N$ be an irreducible and reduced non-degenerate variety and let $\Hilb_{h}(X)$ be the Hilbert scheme parametrizing zero-dimensional subschemes of $X$ of length $h$. 
\begin{Definition}\label{vsph}
Let $p\in\mathbb{P}^{N}$ be a general point. We define 
$$VSP_H^X(p,h)^{o} := \{\{x_{1},...,x_{h}\}\in\Hilb_{h}(X)\: | \: p\in \langle x_{1},...,x_{h}\rangle\}\subseteq \Hilb_{h}(X),$$
and
$$VSP_H^X(p,h) := \overline{VSP_H^X(p,h)^{o}}.$$
by taking the closure of $VSP_H^X(p,h)^{o}$ in $\Hilb_{h}(X)$. When there is no danger of confusion, we write simply $VSP_H^X(h)$ for $VSP_H^X(p,h)$.
\end{Definition}

Let $\nu_{d}^{n}:\mathbb{P}^{n}\rightarrow\mathbb{P}^{N(n,d)}$, with $N(n,d) =\binom{n+d}{d}-1$ be the Veronese embedding induced by $\mathcal{O}_{\mathbb{P}^{n}}(d)$, and let $V_{d}^{n} = \nu_{d}^{n}(\mathbb{P}^{n})$ be the corresponding Veronese variety. Note that when $X = V_{d}^{n}$ we recover the classical variety of sums of powers $VSP_H^X(h) = VSP(F,h)$ parametrizing additive decompositions of a general homogeneous polynomial $F\in k[x_0,...,x_n]_{d}$ as sum of $d$-powers of linear forms, see \cite{Do}.

\begin{Proposition}\label{dim}
Assume the general point $p\in\mathbb{P}^{N}$ to be contained in a $(h-1)$-linear space $h$-secant to $X$. Then the variety $VSP_H^X(h)$ has dimension 
$$\dim(VSP_{H}^{X}(h)) = h(n+1)-N-1.$$
Furthermore if $n = 2$ and $X$ is a smooth surface then for $p$ varying in an open Zariski subset of $\mathbb{P}^{N}$ the varieties $VSP_{H}^{X}(h)$ are smooth and irreducible.
\end{Proposition}
\begin{proof}
Consider the incidence variety 
  \[
  \begin{tikzpicture}[xscale=1.5,yscale=-1.5]
    \node (A0_1) at (1, 0) {$\mathcal{I} = \{(Z,p)\: | \: Z\in VSP_{H}^{X}(h)\}\subseteq \Hilb_{h}(X)\times \mathbb{P}^{N}$};
    \node (A1_0) at (0, 1) {$\Hilb_{h}(X)$};
    \node (A1_2) at (2, 1) {$\mathbb{P}^{N}$};
    \path (A0_1) edge [->]node [auto] {$\scriptstyle{\psi}$} (A1_2);
    \path (A0_1) edge [->]node [auto,swap] {$\scriptstyle{\phi}$} (A1_0);
  \end{tikzpicture}
  \]
The morphism $\phi$ is surjective and there exists an open subset $U\subseteq \Hilb_{h}(X)$ such that for any $Z\in U$ the fiber $\phi^{-1}(Z)$ is isomorphic to the Grassmannian $\mathbb{P}^{h-1}$, so $\dim(\phi^{-1}(Z)) = h-1$. The fibers of $\psi$ are the varieties $VSP_{H}^{X}(h)$. Under our hypothesis the morphism $\psi$ is dominant and 
$$\dim(VSP_{H}^{X}(h)) = \dim(\mathcal{I})-N = h(n+1)-N-1.$$
If $n = 2$ and $X$ is a smooth surface then $\Hilb_{h}(X)$ is smooth. The fibers of $\phi$ over $U$ are open Zariski subsets. So $\mathcal{I}$ is smooth and irreducible. Since the varieties $VSP_{H}^{X}(h)$ are the fibers of $\psi$ we conclude that for $p$ varying in an open Zariski subset of $\mathbb{P}^{N}$ the varieties $VSP_{H}^{X}(h)$ are smooth and irreducible.
\end{proof}

Now, we want to define another compactification of $VSP_H^X(h)^{o}$ in the Grassmannian $\G(h-1,N)$. In order to do this we need to introduce secant varieties. Let $X\subset\P^N$ be an irreducible and reduced non-degenerate variety and let
$$\Gamma_h(X)\subset X\times \dots \times X\times\G(h-1,N)$$
be the reduced closure of the graph of
$$\alpha: X\times  \dots \times X \dasharrow \G(h-1,N),$$
taking $h$ general points to their linear span $\langle x_1, \dots , x_{h}\rangle$. Now, $\Gamma_h(X)$ is irreducible and reduced of dimension $hn$. Let $\pi_2:\Gamma_h(X)\to\G(h-1,N)$ be the natural projection. We denote 
$$\mathcal{S}_h(X):=\pi_2(\Gamma_h(X))\subset\G(h-1,N).$$
Note that $\mathcal{S}_h(X)$ is irreducible and reduced of dimension $hn$. Finally, let
$$\mathcal{I}_h=\{(x,\Lambda) \: | \: x\in \Lambda\} \subset\P^N\times\G(h-1,N)$$
with the projections $\pi_h$ and $\psi_h$ onto the factors.\\
The {\it abstract $h$-secant variety} is the irreducible and reduced variety
$$\Sec_{h}(X):=(\psi_h)^{-1}(\mathcal{S}_h(X))\subset \mathcal{I}_h.$$
The {\it $h$-secant variety} is
$$\sec_{h}(X):=\pi_h(Sec_{h}(X))\subset\P^N.$$
The variety $\Sec_{h}(X)$ has dimension $(hn+h-1)$ and a natural $\P^{h-1}$-bundle structure on $\mathcal{S}_h(X)$. The variety $X$ is \textit{$h$-defective} if $\dim\sec_{h}(X)<\min\{\dim\Sec_{h}(X),N\}$.

\begin{Remark}
Note that in Proposition \ref{dim} the assumption that the general point $p\in\mathbb{P}^{N}$ is contained in a $(h-1)$-linear space $h$-secant to $X$  can be rephrased as $\sec_{h}(X) = \mathbb{P}^{N}$. 
\end{Remark}

We need to extend these notions to the relative case. Let $S$ be a noetherian scheme and let $X\rightarrow S$ be a scheme over $S$ such that there exists a coherent sheaf $E$ on $S$ with a closed embedding of $X$ in $\mathbb{P}(E):= \Proj Sym_{\mathcal{O}_{S}} (E)$ over $S$. Equivalently, we can assume that there exists a relatively ample line bundle $L$ on $X$ over $S$.\\
There exists a scheme $\Grass(h,E)$ that finely parametrizes locally free sub-sheaves of rank $h$ of $E$. Furthermore, $\Grass(h,E)$ is projective over $S$.\\
Now suppose that $E$ is a rank $N+1$ vector bundle and the fiber of the morphism $\Grass(h,E)\rightarrow S$ over a closed point $s\in S$ is the Grassmannian $\Grass(h,E_{s})\iso\G(h,N)$, where $E_{s}$ is the fiber of $E$ over $s\in S$. There is a well-defined rational map over $S$
\[
\begin{tikzpicture}[xscale=2.5,yscale=-1.2]
    \node (A0_0) at (0, 0) {$X\times_{S}  \dots \times_{S} X$};
    \node (A0_2) at (2, 0) {$\Grass(h,E)$};
    \node (A1_1) at (1, 1) {$S$};
    \path (A0_0) edge [->]node [auto] {$\scriptstyle{}$} (A1_1);
    \path (A0_2) edge [->]node [auto] {$\scriptstyle{}$} (A1_1);
    \path (A0_0) edge [->,dashed]node [auto] {$\scriptstyle{\alpha}$} (A0_2);
\end{tikzpicture}
\]
mapping $(x_{1}, \dots ,x_{h})$ to the linear span $\langle x_{1}, \dots ,x_{h}\rangle$. Note that since $\alpha$ is a map over $S$, we are taking $x_{i}\in X_{s}\subset\mathbb{P}(E_{s})\cong\mathbb{P}^{N}$ for some $s\in S$. Take $\Gamma_{h}^{S}(X)$ to be the reduced closure of the graph of $\alpha$ in $X\times_{S} \dots \times_{S}X\times_{S}\Grass(h,E)$; then $\Gamma_{h}^{S}(X)$ is irreducible and reduced of dimension $hn$ over $S$.\\
Let $\pi:\Gamma_{h}^{S}(X)\rightarrow\Grass(h,E)$ be the projection, denoted by
$$\mathcal{S}_{h}^{S}(X):= \pi(\Gamma_{h}^{S}(X))\subseteq\Grass(h,E).$$
Again $\mathcal{S}_{h}^{S}(X)$ is irreducible and reduced of dimension $hn$ over $S$, where $n = \dim_{S}(X)$. Now consider the incidence correspondence
  \[
  \begin{tikzpicture}[xscale=1.5,yscale=-1.2]
    \node (A0_1) at (1, 0) {$\mathcal{I}_{h}^{S} := \{(z,F) \: | \: z\in F\}\subseteq\mathbb{P}(E)\times_{S}\Grass(h,E)$};
    \node (A1_0) at (0, 1) {$\mathbb{P}(E)$};
    \node (A1_2) at (2, 1) {$\Grass(h,E)$};
    \node (A2_1) at (1,2) {$S$};
    \path (A0_1) edge [->]node [auto] {$\scriptstyle{\psi_h}$} (A1_2);
    \path (A0_1) edge [->]node [auto,swap] {$\scriptstyle{\pi_h}$} (A1_0);
    \path (A1_0) edge [->]node [auto] {} (A2_1);
    \path (A1_2) edge [->]node [auto] {} (A2_1);
  \end{tikzpicture}
  \]
Let $X\rightarrow S$ be an irreducible and reduced scheme over $S$, together with a closed embedding in $\P(E)$. The \textit{abstract relative $h$-secant variety} of $X$ over $S$ is
$$\Sec_{h}^{S}(X):= \psi_h^{-1}(\mathcal{S}_{h}^{S}(X))\subseteq \mathcal{I}_{h}^{S}$$
and the \textit{relative $h$-secant variety} of $X$ over $S$ is 
$$\sec_{h}^{S}(X):= \pi_h(\Sec_{h}^{S}(X))\subseteq\mathbb{P}(E).$$
\begin{Remark}\label{fam}
The scheme $\sec_{h}^{S}(X)$ naturally comes with a morphism $\sec_{h}^ {S}(X)\rightarrow S$ whose fiber over a closed point $s\in S$ is the $h$-secant variety $\sec_{h}(X_{s}) \subseteq\mathbb{P}(E_{s}) \cong \mathbb{P}^{N}$ of the fiber $X_{s}$ of $X\rightarrow S$ over $s\in S$.
\end{Remark}

\begin{Definition}\label{vspg} 
Let $X\subset\P^N$ be an irreducible non-degenerate variety of dimension $n$ and let $p\in\P^N$ be a general point. For $h+n<N+1$ consider the $h$-secant map $\pi_h:Sec_{h}(X)\to\P^N$ and define
$$VSP_G^X(p,h):=\pi_h^{-1}(p).$$
When there is no danger of confusion, we write simply $VSP_G^X(h)$ for $VSP_G^X(p,h)$.
\end{Definition}

\begin{Remark}\label{comp}
Let us compare the two compactifications $VSP_H^X(h)$ and $VSP_G^X(h)$. A general point of $VSP_G^X(h)$ corresponds to a $(h-1)$-linear space $h$-secant to $X$ and passing though $p\in\mathbb{P}^{N}$. Clearly there is a dominant rational map
$$
\begin{array}{cccc}
\tau: & VSP_H^X(h) & \dasharrow & VSP_G^X(h)\\
 & \{x_{1},...,x_{h}\} & \longmapsto & \langle x_{1},...,x_{h}\rangle
\end{array}
$$
Furthermore if $h+n<N+1$ the general $(h-1)$-linear space intersects $X$ in a subscheme consisting of $h$ distinct points, so $\tau:VSP_H^X(h)\dasharrow VSP_G^X(h)$ is birational.
\end{Remark}

\begin{Remark}
The inequality $h+n<N+1$ of Definition \ref{vspg} is not a restriction for our purposes. Indeed when we use the compactification $VSP_G^X(h)$ we begin by studying $VSP_G^X(\overline{h})$ for a particular value $\overline{h}$ of $h$ satisfying this inequality and then we extend our conclusions for $h\geq\overline{h}$ using Construction \ref{chain}.
\end{Remark}

\subsubsection*{Rational connectedness}
We say that a variety $X$ is \textit{rationally chain connected} if there is a family of proper and connected algebraic curves $g:U\rightarrow Y$ whose geometric fibers have only rational components with a morphism $\nu:U\rightarrow X$ such that
$$\nu\times\nu:U\times_Y U\rightarrow X\times X$$
is dominant. Note that the image of $\nu\times\nu$ consists of pairs $(x_1,x_2)\in X$ such that $x_1,x_2\in u(\pi^{-1}(y))$ for some $y\in Y$, where $\pi:U\times_Y U\rightarrow Y$ is the projection. We say that $X$ is \textit{rationally connected} if there is a family of proper and connected algebraic curves $g:U\to Y$ whose geometric fibers are irreducible rational curves with morphism $\nu:U\to X$ such that $\nu\times\nu$ is dominant, see \cite[Definition IV.3.2]{Ko}.\\
A proper variety $X$ over an algebraically closed field is rationally chain connected if there is a chain of rational curves through any two general points $x_1,x_2\in X$. The variety $X$ is rationally connected if there is an irreducible rational curve through any two general points $x_1,x_2\in X$. If $X$ is smooth these two notions are indeed equivalent, see \cite[Theorem IV. 3.10]{Ko}. This is clearly false when $X$ is singular. For instance the cone $C_{E}$ over an elliptic curve $E$ is rationally chain connected but it is not rationally connected.\\
Furthermore rational connectedness is a birational property and indeed if $X$ is rationally connected and $X\dasharrow Y$ is a dominant rational map then $Y$ is rationally connected. On the other hand rational chain connectedness is not a birational property. For instance the cone $C_{E}$ is rationally chain connected by chains of length two but its resolution $\widetilde{C}_{E}$ is a fibration over $E$ with fibers $\mathbb{P}^{1}$ so it is not rationally chain connected. Finally for rational fibrations the rational connectedness of the base and of the general fiber translates in the rational connectedness of the variety itself.
\begin{Proposition}\label{rcfib}
Let $\phi:X\dasharrow Y$ be a dominant rational map. If $Y$ and the general fiber of $\phi$ are rationally connected then $X$ is rationally connected.
\end{Proposition}
\begin{proof}
Let us consider a resolution of the indeterminacy 
 \[
  \begin{tikzpicture}[xscale=2.8,yscale=-1.5]
    \node (A0_0) at (0, 0) {$\widetilde{X}$};
    \node (A1_0) at (0, 1) {$X$};
    \node (A1_1) at (1, 1) {$Y$};
    \path (A0_0) edge [->,swap] node [auto] {$\scriptstyle{\pi}$} (A1_0);
    \path (A1_0) edge [->,dashed] node [auto] {$\scriptstyle{\phi}$} (A1_1);
    \path (A0_0) edge [->] node [auto] {$\scriptstyle{\widetilde{\phi}}$} (A1_1);
  \end{tikzpicture}
  \]
of the rational map $\phi$. Now, $\widetilde{\phi}:\widetilde{X}\rightarrow Y$ is a dominant morphism with rationally connected general fiber. Then, by \cite[Corollary 1.3]{GHS} $\widetilde{X}$ is rationally connected. Finally the morphism $\pi:\widetilde{X}\rightarrow X$ is birational and $X$ is rationally connected as well. 
\end{proof}

The rational connectedness of a projective variety $X\subset\mathbb{P}^{N}$ is related to the fact that $X$ has low degree. For instance a smooth hypersurface $X\subset\mathbb{P}^{N}$ of degree $d$ is rationally connected if and only if $d\leq N$. In this context we will need the following theorem.

\begin{Theorem}\cite[Theorem 3.1]{MaM}\label{as}
Let $X\subset\mathbb{P}^{N}$ be a variety set theoretically defined by homogeneous polynomials $G_{i}$ of degree $d_{i}$, for $i = 1,..,m$, and let $l\geq 2$ be an integer. If
$$\sum_{i=1}^{m}d_{i}\leq \frac{N(l-1)+m}{l}$$
then $X$ is rationally chain connected by chains of lines of length at most $l$.\\ 
In particular if $X$ is smooth and the above inequality is satisfied then $X$ is rationally connected by rational curves of degree at most $l$.
\end{Theorem}

Note that the variety $X$ of Theorem \ref{as} is a general projective variety, which does not need to be neither smooth nor a complete intersection. 

\section{Varieties of minimal degree and rationality}\label{smd}
There is a lower bound on the degree of an irreducible, reduced and non-degenerate variety $X\subset\mathbb{P}^{N}$.

\begin{Proposition}
If $X\subset\mathbb{P}^{N}$ is an irreducible, reduced and non-degenerate variety, then $\deg(X)\geq\codim(X)+1$. 
\end{Proposition}
\begin{proof}
If $\codim(X) = 1$, being $X$ non-degenerate, we have $\deg(X)\geq 2 = \codim(X)+1$. We proceed by induction on $\codim(X)$. Let $x\in X$ be a general point, and 
$$\pi_{x}:\mathbb{P}^{N}\dasharrow\mathbb{P}^{N-1}$$
be the projection from $x$. The variety $Y = \overline{\pi_{x}(X)}\subset\mathbb{P}^{N-1}$ has degree $\deg(Y) = \deg(X)-1$, and codimension $\codim(Y) = \codim(X)-1$. By induction hypothesis we have $\deg(Y)\geq\codim(Y)+1$, which implies $\deg(X)\geq\codim(X)+1$.
\end{proof}

\begin{Definition}
We say that an irreducible, reduced and non-degenerate variety $X\subset\mathbb{P}^{N}$ is a \textit{variety of minimal degree} if $\deg(X) = \codim(X)+1$.
\end{Definition}

If $\codim(X) = 1$ then $X$ is a quadric hypersurface, and then classified by its dimension and its singular locus. In higher codimension we have that if $X\subset\mathbb{P}^{N}$ is a singular variety of minimal degree, then $X$ is a cone over a smooth such variety. If $X$ is smooth and $\codim(X)\geq 2$, then $X$ is either a rational normal scroll or the Veronese surface $V^{2}_{2}\subset\mathbb{P}^{5}$. For a survey on varieties of minimal degree see \cite{EH}.

\begin{Theorem}\label{mindeg}
Let $X\subset\mathbb{P}^{N}$ be a variety of minimal degree $\deg(X)=d$. Then $VSP_{H}^{X}(h)$ is irreducible for any $h\geq d$. Furthermore $VSP_{H}^{X}(h)$ is rational if $h = d$, and unirational for any $h\geq d$. 
\end{Theorem}
\begin{proof}
Let $p\in\P^{N}$ be a general point. Since 
$$\dim(X)+(d-1) = N-\codim(X)+d-1 = N$$ 
a general $(d-1)$-plane $\Lambda$ through $p$ intersects $X$ in $d$ distinct points $\Lambda\cap X = \{x_{1},...,x_{d}\}$. Clearly $p\in \Lambda = \langle x_{1},...,x_{d}\rangle$, and $\sec_{d}(X) = \mathbb{P}^{N}$. The $(d-1)$-planes in $\mathbb{P}^{N}$ passing through $p$ are parametrized by the Grassmannian $\G(N-d,N-1)$. Therefore we have a generically injective rational map
$$
\begin{array}{cccc}
\chi: & \G(N-d,N-1) & \dasharrow & VSP_{H}^{X}(d)\\
 & \Lambda & \longmapsto & \Lambda\cap X
\end{array}
$$
Now, it is enough to observe that 
$$\dim(\G(N-d,N-1)) = (N-d+1)(d-1) = n(d-1) = d(n+1)-N-1 = \dim(VSP_{H}^{X}(d))$$ 
to conclude that $VSP_{H}^{X}(d)$ is irreducible and rational.\\
For $h>d$ consider the incidence variety
    \[
  \begin{tikzpicture}[xscale=1.5,yscale=-1.5]
    \node (A0_1) at (1, 0) {$Y= \{((x_{1},\lambda_{1}),...,(x_{h-d},\lambda_{h-d}),\Lambda)\: | \: p-\sum_{i=1}^{h-d}\lambda_{i}x_{i}\in\Lambda\}\subseteq (X\times\mathbb{P}^{1})^{h-d}\times\G(d-1,N)$};
    \node (A1_0) at (0, 1) {$(X\times\mathbb{P}^{1})^{h-d}$};
    \node (A1_2) at (2, 1) {$\G(d-1,N)$};
    \path (A0_1) edge [->]node [auto] {$\scriptstyle{\psi}$} (A1_2);
    \path (A0_1) edge [->]node [auto,swap] {$\scriptstyle{\phi}$} (A1_0);
  \end{tikzpicture}
  \]
The morphism $\phi:Y\rightarrow (X\times\mathbb{P}^{1})^{h-d}$ is surjective and its fibers are isomorphic to the Grassmannian $\G(N-d,N-1)$. Then $Y$ is irreducible. Note that $(X\times\mathbb{P}^{1})^{h-d}$ is rational being $X$ of minimal degree and hence rational. Then the variety $Y$ is rational. Since $\chi$ is birational, for $((x_{1},\lambda_{1}),...,(x_{h-d},\lambda_{h-d}),\Lambda)\in Y$ general the intersection $\Lambda\cap X=\{\hat{x}_{1},...,\hat{x}_{d}\}$ determines a decomposition $p-\sum_{i=1}^{h-d} = \sum_{j=1}^{d}\hat{\lambda}_{j}\hat{x}_{j}$. The map 
$$
\begin{array}{cccc}
\alpha: & Y & \dasharrow & VSP_{H}^{X}(h)\\
 & ((x_{1},\lambda_{1}),...,(x_{h-d},\lambda_{h-d}),\Lambda) & \longmapsto & \{x_{1},...,x_{h-d},\hat{x}_{1},...,\hat{x}_{d}\}
\end{array}
$$
is a generically finite, rational map, of degree $\binom{h}{h-d}$. Now, it is enough to observe that
$$\dim(Y) = (n+1)(h-d)+(N-d+1)(d-1) = h(n+1)-N-1 = \dim(VSP_{H}^{X}(h))$$
to conclude that $\alpha$ is dominant. The variety $VSP_{H}^{X}(h)$ is finitely dominated by a rational variety, then it is irreducible and unirational. 
\end{proof}

\begin{Example}
Let $Q\subset\P^{3}$ be a smooth quadric. Since any line through a general point $p\in\P^{3}$ cuts on $Q$ a length two zero-dimensional subscheme the morphism 
$$\chi:\P^{2}\rightarrow VSP_{H}^{Q}(2)$$
is an injective regular morphism. By Proposition \ref{dim} $VSP_{H}^{Q}(2)$ is a smooth surface, so $\chi$ is an isomorphism and $VSP_{H}^{Q}(2)\cong\P^{2}$. More generally for a smooth quadric hypersurface $Q\subset\mathbb{P}^{N}$ we have a birational morphism
$$\P^{N-1}\rightarrow VSP_{H}^{Q}(2)$$
and $VSP_{H}^{Q}(2)$ is rational.
\end{Example}

\section{Low degree and rational connectedness}\label{ldrc}
In this section we will prove two results relating the rational connectedness of $VSP_{H}^{X}(h)$ to the fact that the base variety $X\subset\mathbb{P}^{N}$ has low degree. In the following we will always assume that $\sec_{h}(X) = \mathbb{P}^{N}$.

\begin{Convention}\label{conv}
When we refer to a general decomposition we always consider the irreducible component of $VSP_{H}^{X}(h)^{o}$ containing it. We still denote by $VSP_{H}^{X}(h)$ its compactification.
\end{Convention}

\subsubsection*{Fibrations approach}
Let us begin with the simple case of a non-degenerate hypersurface $X\subset\mathbb{P}^{N}$ of degree $d$. Let $\G(h-2,N-1)$ be the Grassmannian of $(h-1)$-planes in $\mathbb{P}^{N}$ passing through a general point $p\in\mathbb{P}^{N}$. We have a rational dominant map
$$
\begin{array}{cccc}
\chi: & VSP_{H}^{X}(h) & \dasharrow & \G(h-2,N-1)\\
 & \{x_1,...,x_h\} & \longmapsto & \left\langle x_1,...,x_h\right\rangle
\end{array}
$$
Let $H\in \G(h-2,N-1)$ be a general point, $X_{H} = X\cap H$ and  
$$X_{H}^{h}:=\underbrace{X_{h}\times...\times X_h}_{h\; \rm times}.$$
The fiber $\chi^{-1}(H)$ is birational to $X_{H}^{h}/S_{h}$.
Now, $X_{H}\subset H\cong\mathbb{P}^{h-1}$ is a hypersurface of degree $d$. Therefore $X_{H}$ is rationally connected if and only if $d\leq h-1$. In this case $X_{H}^{h}$ and $\chi^{-1}(H)$ are rationally connected as well. So $\chi$ is a rational dominant map on a rational variety with rationally connected general fiber. By Proposition \ref{rcfib} we conclude that if $d\leq h-1$ the $VSP_{H}^{X}(h)$ is rationally connected. In the following we generalize this argument.

\begin{Theorem}\label{rcdeg}
Let $X\subset\mathbb{P}^{N}$ be an irreducible variety. Assume $h > N-\dim(X)+1$ and that the general $(h-1)$-dimensional linear section of $X$ is rationally connected. Then the irreducible components of $VSP_{H}^{X}(h)$ are rationally connected.
\end{Theorem}
\begin{proof}
As before we have a rational map
$$
\begin{array}{cccc}
\chi: & VSP_{H}^{X}(h) & \dasharrow & \G(h-2,N-1)\\
 & \{x_1,...,x_h\} & \longmapsto & \left\langle x_1,...,x_h\right\rangle
\end{array}
$$
Since $h>N-\dim(X)+1$, if $H\in\G(h-2,N-1)$ is an $(h-1)$-plane through $p$ and $X_{H} = X\cap H$, then $\dim(X_{H})\geq 1$ and $\chi$ is dominant.\\
Now, if $H\in\G(h-2,N-1)$ is general then $X_{H}\subset H\cong\mathbb{P}^{h-1}$ is a smooth, rationally connected variety. The general fiber $\chi^{-1}(H)$ is birational to $X_{H}^{h}/S_h$, so it is rationally connected as well. Finally, by Proposition \ref{rcfib} we conclude that $VSP_{H}^{X}(h)$ is rationally connected.
\end{proof}

An immediate consequence of Theorem \ref{rcdeg} is the following.

\begin{Proposition}
Let $X\subset\mathbb{P}^{N}$ be a smooth variety set theoretically defined by homogeneous polynomials $G_{i}$ of degree $d_{i}$, for $i = 1,..,m$, and let $l\geq 2$ be an integer. Assume $h > N-\dim(X)+1$. If
$$\sum_{i=1}^{m}d_{i}\leq \frac{(h-1)(l-1)+m}{l}$$
then the irreducible components of $VSP_{H}^{X}(h)$ are rationally connected.\\ 
Furthermore if $X\subset\mathbb{P}^{N}$ is a smooth complete intersection of $G_{1},...,G_{c}$ hypersurfaces of degree $d_{i}$ for $i = 1,..,c$ and
$$\sum_{i}^{c}d_{i}\leq h-1$$
then the irreducible components of $VSP_{H}^{X}(h)$ are rationally connected.
\end{Proposition}
\begin{proof}
If $H\in\G(h-2,N-1)$ is general then $X_{H}\subset H\cong\mathbb{P}^{h-1}$ is a smooth variety defined by homogeneous polynomials $\widetilde{G}_{i}:= G_{i|H}$ of degree $d_{i}$, for $i = 1,..,m$. By Theorem \ref{as} under our numerical hypothesis the variety $X_{H}$ is rationally connected. To conclude it is enough to apply Theorem \ref{rcdeg}.\\
If $X$ is a smooth complete intersection then $X_{H}\subset\mathbb{P}^{h-1}$ is the smooth complete intersection of the hypersurfaces $\widetilde{G}_{i}:= G_{i|H}$ of degree $d_{i}$, for $i = 1,..,c$. The second statement follows from the fact that $X_{H}$ is rationally connected if and only if $\sum_{i}^{c}d_{i}\leq h-1$ and from Theorem \ref{rcdeg}.
\end{proof}

\begin{Example}
Let $X\subset\mathbb{P}^{9}$ be a complete intersection of two general hypersurfaces of degree two. Then the irreducible components of $VSP_{H}^{X}(h)$ are rationally connected for any $h\geq 5$.
\end{Example}

\begin{Remark}
Recall that we have a dominant rational map
$$
\begin{array}{cccc}
\tau: & VSP_H^X(h) & \dasharrow & VSP_G^X(h)\\
 & \{x_{1},...,x_{h}\} & \longmapsto & \langle x_{1},...,x_{h}\rangle
\end{array}
$$
Therefore Theorem \ref{rcdeg} holds for the compactification $VSP_G^X(h)$ as well.
\end{Remark}

By the argument used in the proof of Theorem \ref{rcdeg} we get the following result for the classical varieties of sums of powers $VSP(F,h)$ when $X = V_{d}^{n}\subset\mathbb{P}^{N(n,d)}$ is a Veronese variety.

\begin{Proposition}\label{ver}
Let $VSP(F,h)$ be the variety of sums of powers of a general homogeneous polynomial $F\in k[x_0,...,x_n]_{d}$. If
$$h\geq \frac{d(N+1)-n}{d}$$
then the irreducible components of $VSP(F,h)$ are rationally connected. 
\end{Proposition}
\begin{proof}
The numerical hypothesis implies $h>N-n+1$. Then the rational map
$$
\begin{array}{cccc}
\chi: & VSP(F,h) & \dasharrow & \G(h-2,N-1)\\
 & \{L_1,...,L_h\} & \longmapsto & \left\langle L_1,...,L_h\right\rangle
\end{array}
$$
is dominant. In this case $X = V_{d}^{n}$ and the general fiber of $\chi^{-1}(H)$ is birational to $X_{H}^{h}/S_h$ where $X_{H}$ is a smooth complete intersection of $N-h+1$ hypersurfaces of degree $d$ in $\mathbb{P}^{n}$. Therefore
$$\omega_{X_{H}}\cong\mathcal{O}_{\mathbb{P}^{n}}((N-h+1)d-n-1).$$  
To conclude it is enough to observe that under our numerical hypothesis $X_{H}$ is Fano and therefore rationally connected and to apply Proposition \ref{rcfib}.
\end{proof}

\begin{Example}
If $d=n=3$, then $N=19$. The bound of Proposition \ref{ver} is $h\geq 19$. For instance when $h = 19$ we have the fibration
$$
\begin{array}{cccc}
\chi: & VSP(F,19) & \dasharrow & \G(17,18)\cong\mathbb{P}^{18}\\
 & \{L_1,...,L_{19}\} & \longmapsto & \left\langle L_1,...,L_{19}\right\rangle
\end{array}
$$
whose general fiber is a cubic surface. Note that this case is not covered by \cite[Theorem 4.1]{MMe}.
\end{Example}
In the following table we work out some numbers which make Proposition \ref{ver} work. We denote by $\overline{h}$ the smallest $h$ for which Proposition \ref{ver} holds. Clearly this result makes sense for $n\gg d$.
\begin{center}
{\small
\begin{tabular}{|c|c|c|c|}
\hline
$d$ & $n$ & $N$ & $\overline{h}$\\
\hline
3 & 100 & 176850 & 176818\\
3 & 150 & 585275 & 585226\\
4 & 200 & 70058750 & 70058701\\ 
\hline
\end{tabular}
}
\end{center}

\subsubsection*{Chains approach}
In the following our aim is to prove the rational connectedness of another class of generalized varieties of sums of powers by extending the approach of \cite[Section 4]{MMe} from the Veronese varieties to arbitrary unirational varieties. In this section we will use the notion of very general point. Therefore we assume that ground field $k$ is uncountable.

\begin{Construction}\label{chain}
Assume $VSP_{H}^{X}(h)$ to be non-empty, and let $\{x_{1},...,x_{h}\}\in VSP_{H}^{X}(h)$ be a general point. Then there exists a general point $p\in \mathbb{P}^{N}$ such that
$$p = \sum_{i=1}^{h}\lambda_{i}x_{i}.$$
The point $p-\lambda_{1}x_{1}$ is general as well, and we get a generically injective rational map
$$
\begin{array}{ccc}
VSP_{H}^{X}(h-1) & \dasharrow &  VSP_{H}^{X}(h)\\
 \{y_{1},...,y_{h-1}\} & \longmapsto & \{x_{1},y_{1},...,y_{h-1}\}
\end{array}
$$
This construction yields a stratification
$$VSP_{H}^{X}(h-r)\subset VSP_{H}^{X}(h-r+1)\subset ...\subset VSP_{H}^{X}(h-1)\subset VSP_{H}^{X}(h).$$
\end{Construction}

\begin{Proposition}\label{p2}
Let $X\subset\mathbb{P}^{N}$ be a non-degenerate variety of dimension $n$. If 
$$h\geq \frac{(N+1)}{n+1}+2$$
then two very general points of $VSP_{H}^{X}(h)$ are joined by a chain, of at most length three, of $VSP_{H}^{X}(h-1)$. If $V_{i}$ are the elements of this chain and $q\in V_{i}\cap V_{j}$ is a general point, then we can assume $q$ to be a smooth point in $V_{i},V_{j}$ and $VSP_{H}^{X}(h)$.
\end{Proposition}
\begin{proof}
Let $x=\{x_{i}\},y = \{y_{i}\}\in VSP_{H}^{X}(h)$ be two very general points, and write
$$p = \sum_{i=1}^{h}\lambda_{i}x_{i} = \sum_{i=1}^{h}\gamma_{i}y_{i}.$$
Let $z\in VSP_{H}^{X}(p-\lambda_{1}x_{1},h-1)$ be a general point associated to the decomposition
$$p-\lambda_{1}x_{1} = \sum_{i=2}^{h}\alpha_{i}z_{i}.$$
Let $\nu:Z\rightarrow VSP_{H}^{X}(h)$ be a resolution of singularities. Since $x$ and $y$ are two very general points we can assume that 
\begin{itemize}
\item[-] $\nu^{-1}(VSP_{H}^{X}(p-\lambda_{1}x_{1},h-1))$ and $\nu^{-1}(VSP_{H}^{X}(p_{j}-\gamma_{1}y_{1},h-1))$ belong to the same irreducible component of $\Hilb(Z)$.
\item[-] $\nu$ is an isomorphism in a neighborhood of $q$.
\end{itemize}
Since $z\in VSP_{H}^{X}(h)$ is associated to $p = \lambda_{1}x_{1}+\sum_{i=2}^{h}\alpha_{i}z_{i}$ we have
$$z\in VSP_{H}^{X}(p-\lambda_{1}x_{1},h-1)\cap VSP_{H}^{X}(p-\alpha_{2}z_{2},h-1).$$
Under our numerical hypothesis we have 
$$\dim(VSP_{H}^{X}(p-\alpha_{2}z_{2},h-1))\geq\codim_{VSP_{H}^{X}(h)}(VSP_{H}^{X}(p-\alpha_{2}z_{2},h-1)),$$
and we conclude that 
$$VSP_{H}^{X}(p-\alpha_{2}z_{2},h-1)\cap VSP_{H}^{X}(p-\gamma_{1}y_{1},h-1)\neq\emptyset.$$
Furthermore the general point of this intersection is a smooth point of $VSP_{H}^{X}(p-\alpha_{2}z_{2},h-1)$, $VSP_{H}^{X}(p-\gamma_{1}y_{1},h-1)$ and $VSP_{H}^{X}(h)$. Finally we can join two general points of $VSP_{H}^{X}(h)$ by a chain of length at most three of $VSP_{H}^{X}(h-1)$.
\end{proof}

In what follows our aim is to generalize \cite[Theorem 4.1]{MMe} taking an arbitrary unirational variety $X\subset\mathbb{P}^{N}$ instead of the Veronese variety $V_{d}^{n}\subset\mathbb{P}^{N(n,d)}$. The first step of our argument is to prove that a unirational variety can be covered by rational subvarieties in a suitable sense.

\begin{Proposition}\label{p1}  
Let $X$ be a unirational variety. For any triple of integers
$(a,b,c)$, with $0<c<n$, there is a rationally connected variety $V^n_{a,b,c}\subset\Hilb(X)$ with the following properties:
\begin{itemize}
\item[-] a general point in $V^n_{a,b,c}$ represents  a rational subvariety of $X$ of codimension $c$;
\item[-] for a general $Z\subset X$ reduced zero dimensional scheme of length $l\leq b$, there is a rationally connected subvariety $V_{Z,c}\subset V^n_{a,b,c}$, of dimension at least $a$, whose general element $[Y]\in V_{Z,c}$ represents a rational  subvariety of $X$ of codimension $c$ containing $Z$.
\end{itemize}
\end{Proposition}
\begin{proof}
Since $X$ is unirational there is a generically finite, dominant map $\phi:\mathbb{P}^{n}\dasharrow X$. For any Hilbert polynomial $P\in\mathbb{Q}[z]$ the map $\phi$ induces a generically finite rational map
$$
\begin{array}{cccc}
\chi: & \Hilb_{P}(\P^{n}) & \dasharrow & \Hilb_{Q}(X)\\
 & Z & \longmapsto & \phi(Z)
\end{array}
$$
We prove the statement  by induction on $c$. Assume $c=1$, and consider an equation of the form 
$$Y=(x_nA(x_0,\ldots,x_{n-1})_{d-1}+B(x_0,\ldots,x_{n-1})_d=0),$$
then, for $A$ and $B$ general, $Y\subset\mathbb{P}^{n}$ is a rational hypersurface of degree $d$ with a unique singular point of multiplicity $d-1$ at the point $[0,\ldots,0,1]$. Take $A$ and $B$ general. Let $\overline{Y}:=\overline{\phi(Y)}$ be the closure of the image of $Y$ in $X$. If $\overline{y}\in\overline{Y}$ is a general point the fiber $\phi^{-1}(\overline{y})$ intersects $Y$ in a point, that is $\phi_{|Y}:Y\rightarrow\overline{Y}$ is birational.\\ 
Fix $d>ab$ and let $W^n_{a,b,1}\subset \P(k[x_0,\ldots,x_n]_d)$ be the linear span of these
hypersurfaces. We take $V^n_{a,b,1}:= \chi(W^n_{a,b,1})$. 
Let $Z = \{x_{1},...,x_{l}\}\subset X$ be a zero dimensional subscheme of length $l\leq b$, and take $p_{i} \in \phi^{-1}(x_{i})$ for $i = 1,...,l$.\\ 
For any triple $(a,b,1)$ consider $W_{Z,1}\subset W^n_{a,b,1}$ as the sublinear system of hypersurfaces containing $\{p_{1},...,p_{l}\}$. Now take $V_{Z,1}:= \chi(W_{Z,1})$. Then on a general point $[Y]\in W_{Z,1}$ the map $\phi$ restricts to a birational map and a general point of $V_{Z,1}$ parametrizes a rational subvariety of codimension $1$ in $X$ containing $Z$.\\
Assume, by induction, that
$W^n_{a,b,i-1}\subset\Hilb(\P^{n-1})$ exist for any $n$ and $b$.
Define, for $i\geq 2$,
$$\widetilde{W}^n_{a,b,i}:= W^n_{a,b,1}\times
W^{n-1}_{a,b,i-1}\subset\Hilb(\P^n)\times\Hilb(\P^{n-1}).$$ 
Let $[Y]$ be a general point in
$W^n_{a,b,1}$. By construction $Y$ has a point of
multiplicity $d-1$ at the point
$[0,\ldots,0,1]\in \P^n$. Then the
projection
$\pi_{[0,\ldots,0,1]}:\P^n\rat\P^{n-1}$
restricts to a birational map $\f_Y:Y\rat
\P^{n-1}$. Hence we may
associate the general element
$([Y],[S])\in\{[Y]\}\times
W^{n-1}_{a,b,i-1}$ to the
codimension $i$ subvariety
$\f_Y^{-1}(S)\subset\P^n$.
By \cite[Proposition
I.6.6.1]{Ko} this yields a rational map 
$$
\begin{array}{cccc}
\alpha: & \widetilde{W}^n_{a,b,i} & \dasharrow & \Hilb(\P^n)\\
 & (Y,S) & \longmapsto & \f_Y^{-1}(S)
\end{array}
$$
Let $W^n_{a,b,i}:=\overline{\alpha(\widetilde{W}^n_{a,b,i})}\subset\Hilb(\P^n).$
For any $Z$ we may define 
$$\widetilde{W}_{Z,i}:=W_{Z,1}\times
W_{\pi_{[1,0,\ldots,0]}(Z),i-1},$$
and as above $W_{Z,i}=\overline{\alpha(\widetilde{W}_{Z,i})}$.\\
By construction a general point of $W^n_{a,b,c}$ is the inverse image of a rational subvariety of codimension $c-1$ in $\mathbb{P}^{n-1}$ via the projection from the singular point of a general rational hypersurface in $W^n_{a,b,1}$. Then on the general subvariety parametrized by $W^n_{a,b,c}$ and $\widetilde{W}_{Z,c}$ the map $\phi$ restricts to a birational map. We take $V^{n}_{a,b,c}:=\chi(W^n_{a,b,c})$ and $V_{Z,c}:=\chi(\widetilde{W}_{Z,c})$. The varieties $V^{n}_{a,b,c}$ and $V_{Z,c}$ are dominated by rationally connected varieties, so they are rationally connected as well. 
\end{proof}

\begin{Remark}\label{projver}
Let $X\subset\mathbb{P}^{N}$ be a rational, non-degenerate variety of dimension $n$, and let $\phi:\mathbb{P}^{n}\dasharrow X$ be a birational map. The linear system $\mathcal{H} = \phi^{*}\mathcal{O}_{\mathbb{P}^{N}}(1)$ is a subsystem of $\mathcal{O}_{\mathbb{P}^{n}}(d)$ for some integer $d$.  
We can embed $\mathbb{P}^{n}$ via the Veronese embedding $\nu_{d}^{n}$ in $\mathbb{P}^{N(n,d)}$. The variety $X$ is a birational projection 
   \[
  \begin{tikzpicture}[xscale=2.5,yscale=-1.2]
    \node (A0_0) at (0, 0) {$\mathbb{P}^{n}$};
    \node (A0_1) at (1, 0) {$V_{d}^{n}\subset\mathbb{P}^{N(n,d)}$};
    \node (A1_1) at (1, 1) {$X\subset\mathbb{P}^{N}$};
    \path (A0_0) edge [->]node [auto] {$\scriptstyle{\nu_{d}^{n}}$} (A0_1);
    \path (A0_1) edge [->,dashed]node [auto] {$\scriptstyle{\pi}$} (A1_1);
    \path (A0_0) edge [->,dashed]node [auto] {$\scriptstyle{}$} (A1_1);
  \end{tikzpicture}
  \]
of $V_{d}^{n}$. This means that any rational variety can be seen as a birational projection of a suitable Veronese variety.
\end{Remark}

Let $X\subset\mathbb{P}^{N}$ be a unirational variety of dimension $n$. In the following, by Convention \ref{conv} we assume that $VSP_{G}^{X}(h)$ is irreducible. Let us assume there exists an integer $k$ such that
\begin{itemize}
\item[-] $0< k < n$,
\item[-] $\overline{h} = \frac{N}{k+1}$ is an integer,
\item[-] $\overline{h}+n < N+1$.
\end{itemize}
\begin{Remark}
The last condition is required in Definition \ref{vspg} because we are working with the compactification $VSP_{G}^{X}(h)$. Note that the inequality $\overline{h}+n < N+1$ means that $X\subset\mathbb{P}^{N}$ has big codimension. On the other hand low codimension varieties are studied using the fibrations approach of Theorem \ref{rcdeg}. Finally, note that when $X = V_{d}^{n}$ is a Veronese variety the inequality $\overline{h}+n < N+1$ is automatically satisfied. 
\end{Remark}
Let $H_{x}=\left\langle x_1,...,x_{\overline{h}}\right\rangle$ and $H_{y}=\left\langle y_1,...,y_{\overline{h}}\right\rangle$ be two general points of $VSP_{G}^{X}(h)$. In the notation of Proposition \ref{p1} let us consider $V^{n}_{a,2\overline{h},n-k}\subset\Hilb(X)$ for $a\gg 0$ and let $Y\in V^{n}_{a,2\overline{h},n-k}$ be a general point. Recall that:
\begin{itemize}
\item[-] $Y$ is a rational subvariety of $X$ of codimension $n-k$,
\item[-] if $Z = \{x_1,...,x_{\overline{h}},y_1,...,y_{\overline{h}}\}$ then there exists a rationally connected subvariety $V_{Z,n-k}\subset V^{n}_{a,2\overline{h},n-k}$ of dimension at least $a$ whose general point represents a rational subvariety of $X$ of codimension $n-k$ containing $Z$.
\end{itemize}

\begin{Lemma}\label{l1}
The $\overline{h}$-secant variety $\sec_{\overline{h}}(Y)\subset\mathbb{P}^{N}$ is a hypersurface. Furthermore through a general point of $\sec_{\overline{h}}(Y)$ there is a unique $(\overline{h}-1)$-plane $\overline{h}$-secant to $Y$ and $\Sing(\sec_{\overline{h}}(Y))$ has codimension one in $\sec_{\overline{h}}(Y)$.
\end{Lemma}
\begin{proof}
Since $Y$ is rational, by Remark \ref{projver}, we can see $Y$ as a birational projection $\pi:V_{d}^{k}\dasharrow Y$ of a Veronese variety $V_{d}^{k}\subset\mathbb{P}^{N(k,d)}$ for $d\gg \overline{h}$. We denote by $\mathcal{B}\subset V_{d}^{k}$ the indeterminacy locus of $\pi$.\\
Let $z\in\sec_{\overline{h}}(V_{d}^{k})$ be a general point. Assume $z\in \left\langle x_{1},...,x_{\overline{h}}\right\rangle\cap\left\langle y_{1},...,y_{\overline{h}}\right\rangle$. By Terracini's lemma \cite[Theorem 1.1]{CC} we have
$$\mathbb{T}_{z}\sec_{\overline{h}}(V_{d}^{k}) = \left\langle \mathbb{T}_{x_{1}}V_{d}^{k},...,\mathbb{T}_{x_{\overline{h}}}V_{d}^{k}\right\rangle = \left\langle \mathbb{T}_{y_{1}}V_{d}^{k},...,\mathbb{T}_{y_{\overline{h}}}V_{d}^{k}\right\rangle.$$
Then the general hyperplane section singular at $\{x_{1},...,x_{\overline{h}}\}$ is singular at $\{y_{1},...,y_{\overline{h}}\}$.  Since $\codim (\sec_{\overline{h}}(V_{d}^{k}))\geq k+1$ by \cite[Corollary 4.5]{Me2} $V_{d}^{k}$ is not $\overline{h}$-weakly defective. Therefore by \cite[Theorem 1.4]{CC} the general hyperplane section singular at $\{x_{1},...,x_{\overline{h}}\}$ is singular only at $\{x_{1},...,x_{\overline{h}}\}$. So $\{x_{1},...,x_{\overline{h}}\} = \{y_{1},...,y_{\overline{h}}\}$ and through a general point of $\sec_{\overline{h}}(V_{d}^{k})$ there is a unique $(\overline{h}-1)$-plane $\overline{h}$-secant to $V_{d}^{k}$. We conclude that
$$\dim(\sec_{\overline{h}}(V_{d}^{k})) = k\overline{h}+\overline{h}-1 = N-1.$$
Let us assume that $\pi:\sec_{\overline{h}}(V_{d}^{k})\dasharrow\sec_{\overline{h}}(Y)$ is not birational. The same argument of the first part of the proof shows that $V_{d}^{k}$ is $\overline{h}$-weakly defective. This contradicts \cite[Corollary 4.5]{Me2}. Therefore $\sec_{\overline{h}}(Y)$ is a birational projection of $\sec_{\overline{h}}(V_{d}^{k})$. So through a general point of $\sec_{\overline{h}}(Y)$ there is a unique $(\overline{h}-1)$-plane $\overline{h}$-secant to $Y$ and $\dim(\sec_{\overline{h}}(Y)) = N-1$.\\
Now, for a general $Y\in V^{n}_{a,2\overline{h},n-k}$ we have $\left\langle\sec_{\overline{h}}(Y)\cap\left\langle\mathcal{B}\right\rangle\right\rangle\subsetneqq\left\langle\mathcal{B}\right\rangle$. Then we can factorize $\pi$ as a composition of projections 
\[
  \begin{tikzpicture}[xscale=2.5,yscale=-1.2]
    \node (A0_0) at (0, 0) {$\mathbb{P}^{N(k,d)}$};
    \node (A1_0) at (0, 1) {$\mathbb{P}^{N+1}$};
    \node (A1_1) at (1, 1) {$\mathbb{P}^{N}$};
    \path (A1_0) edge [->,dashed] node [auto] {$\scriptstyle{\pi_{2}}$} (A1_1);
    \path (A0_0) edge [->,dashed,swap] node [auto] {$\scriptstyle{\pi_{1}}$} (A1_0);
    \path (A0_0) edge [->,dashed] node [auto] {$\scriptstyle{\pi}$} (A1_1);
  \end{tikzpicture}
  \]
where $\pi_{2}$ is a projection from a general point of $\mathbb{P}^{N+1}$. We know that $\sec_{2}(\pi_{1}(\sec_{\overline{h}}(V_{d}^{k}))) = \mathbb{P}^{N+1}$. Therefore the dimension of the singular locus of $\sec_{\overline{h}}(Y)$ is the dimension of the space of secant lines to $\pi_{1}(\sec_{\overline{h}}(V_{d}^{k}))$ passing through $p$. That is
$$\dim(VSP^{\pi_{1}(\sec_{\overline{h}}(V_{d}^{k}))}_{H}(2))=2(N-1+1)-(N+1)-1 = N-2.$$
\end{proof}

Let $m$ be the degree of the hypersurface $\sec_{\overline{h}}(Y)\subset\mathbb{P}^{N}$. 

\begin{Lemma}
There exists a generically injective rational map
$$
\begin{array}{cccc}
\alpha: & V^{n}_{a,2\overline{h},n-k} & \dasharrow & \mathbb{P}(k[z_{0},...,z_{N}]_{m})\\
 & Y & \longmapsto & \sec_{\overline{h}}(Y)
\end{array}
$$
\end{Lemma}
\begin{proof}
Let $Y\in V^{n}_{a,2\overline{h},n-k}$ be a general point. By Lemma \ref{l1} we have 
$$\sec_{\overline{h}}(Y)\in \mathbb{P}(k[z_{0},...,z_{N}]_{m}).$$ 
Furthermore by Remark \ref{fam} assigning to a general point $Y\in V^{n}_{a,2\overline{h},n-k}$ its $\overline{h}$-secant variety $\sec_{\overline{h}}(Y)$ gives rise to a well defined rational map.\\ 
Now, we want to prove that $\alpha$ is generically injective. Let $Y_{1}\in V^{n}_{a,2\overline{h},n-k}$ be a general point and let $Y_{2}\in V^{n}_{a,2\overline{h},n-k}$ such that $\alpha(Y_{1}) = \sec_{\overline{h}}(Y_{1}) = \sec_{\overline{h}}(Y_{2}) = \alpha(Y_{2})$. Let us denote $S:=\sec_{\overline{h}}(Y_{1}) = \sec_{\overline{h}}(Y_{2})$. Now, $Y_1$ and $Y_2$ are both rational and by Remark \ref{projver} we can see them as projections of two Veronese variety $V_{d}^{k}$ and $V_{d^{'}}^{k}$ for $d,d^{'}\gg 0$.
 
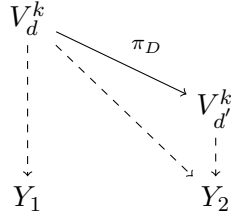
\[
  \begin{tikzpicture}[xscale=2.5,yscale=-1.2]
    \node (A0_0) at (0, 0) {$V_{d}^{k}$};
    \node (A1_1) at (1, 1) {$V_{d^{'}}^{k}$};
    \node (A2_0) at (0, 2) {$Y_{1}$};
    \node (A2_1) at (1, 2) {$Y_{2}$};
    \path (A1_1) edge [->,dashed] node [auto] {$\scriptstyle{}$} (A2_1);
    \path (A0_0) edge [->,dashed] node [auto] {$\scriptstyle{}$} (A2_0);
    \path (A0_0) edge [->,dashed] node [auto] {$\scriptstyle{}$} (A2_1);
    \path (A0_0) edge [->] node [auto] {$\scriptstyle{\pi_{D}}$} (A1_1);
  \end{tikzpicture}
  \]
We can assume $d\geq d^{'}$. If $d>d^{'}$ we can consider a general divisor $D\in |\mathcal{O}_{\mathbb{P}^{k}}(d-d^{'})|$ and interpret $\pi_{D}:V_{d}^{k}\rightarrow V_{d^{'}}^{k}$ as the projection from $\left\langle\nu_{k,d}(D)\right\rangle$. Note that $\pi_{D}$ is well defined on $V_{d}^{k}$. This means that we can interpret both $Y_{1}$ and $Y_{2}$ as projections of the same Veronese variety $V_{d}^{k}$. We summarize the situation in the following diagram 
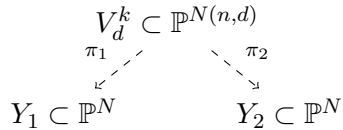
\[
  \begin{tikzpicture}[xscale=1.5,yscale=-1.2]
    \node (A0_0) at (0, 0) {};
    \node (A0_1) at (1, 0) {$V_{d}^{k}\subset\mathbb{P}^{N(n,d)}$};
    \node (A1_0) at (0, 1) {$Y_{1}\subset\mathbb{P}^{N}$};
    \node (A1_2) at (2, 1) {$Y_{2}\subset\mathbb{P}^{N}$};
    \path (A0_1) edge [->,dashed,swap] node [auto] {$\scriptstyle{\pi_{1}}$} (A1_0);
    \path (A0_1) edge [->,dashed] node [auto] {$\scriptstyle{\pi_{2}}$} (A1_2);
  \end{tikzpicture}
  \]
where $\pi_{1}$ and $\pi_{2}$ are projections from linear spaces $\Lambda_{1}$ and $\Lambda_{2}$ respectively. We get two projections
$$\pi_{1}:\sec_{\overline{h}}(V_{d}^{k})\dasharrow S, \: \pi_{2}:\sec_{\overline{h}}(V_{d}^{k})\dasharrow S$$
and the composition $\gamma = \pi_{2}\circ\pi_{1}^{-1}$ is a birational automorphism of $S$. By Lemma \ref{l1} $\Sing(S)$ has codimension one. So $\gamma$ is defined on the general point of $\Sing(S)$. If $u\in\Sing(S)$ is a general point and $v,w\in\pi_{1}^{-1}(u)$ are two general points then $\pi_{2}(v) =\pi_{2}(w)\in \Sing(S)$. Therefore the line $\left\langle v,w\right\rangle$ intersects both $\Lambda_{1}$ and $\Lambda_{2}$. Furthermore $\Lambda_{1}\cap\left\langle v,w\right\rangle = \Lambda_{2}\cap \left\langle v,w\right\rangle$. Proceeding recursively we conclude that $\Lambda_{1} = \Lambda_{2}$ and therefore $Y_{1} = Y_{2}$. 
\end{proof}

Now we are ready to prove the following theorem.

\begin{Theorem}\label{RC2}
Let $X\subset\mathbb{P}^{N}$ be a unirational variety. Assume that for some positive integer $k<n$ the number $\overline{h} = \frac{N}{k+1}$ is an integer and 
$$\frac{N+n+2}{n+1} \leq\overline{h} < N-n+1.$$ 
Then the irreducible components of $VSP_{H}^{X}(h)$ are rationally connected for $h\geq \overline{h}$.
\end{Theorem}
\begin{proof}
We prove the statement for $h = \overline{h}$. Note that $\overline{h}\geq \frac{N+n+2}{n+1}$ implies 
$$\overline{h}+1\geq\frac{N+1}{n+1}+2.$$
Therefore in order to conclude for $h\geq\overline{h}$ we can apply Proposition \ref{p2}.\\ 
Let $V = \overline{\alpha(V^{n}_{a,2\overline{h},n-k})}$ be the closure of the image of $V^{n}_{a,2\overline{h},n-k}$ and let $H_{p}$ be the hyperplane parametrizing hypersurfaces in $\mathbb{P}(k[z_{0},...,z_{N}]_{m})$ passing though the general point $p\in\mathbb{P}^{N}$. We consider the intersection $V_{p}=V\cap H_{p}$ parametrizing $\overline{h}$-secant varieties through $p$. By Lemma \ref{l1} through $p\in\sec_{\overline{h}}(Y)$ there is a unique $(\overline{h}-1)$-plane $H_{p}^{Y}$ which is $\overline{h}$-secant to $Y$. Then we can define the rational map
$$
\begin{array}{cccc}
\beta: & V_{p} & \dasharrow & VSP_{G}^{X}(\overline{h})\\
 & \sec_{\overline{h}}(Y) & \longmapsto & H_{p}^{Y}
\end{array}
$$
By Proposition \ref{p1} a general $\overline{h}$-secant linear space to $X$ is $\overline{h}$-secant to some $Y\in V^{n}_{a,2\overline{h},n-k}$. Then the map $\beta$ is dominant. Let $H_{x}=\left\langle x_1,...,x_{\overline{h}}\right\rangle$ and $H_{y}=\left\langle y_1,...,y_{\overline{h}}\right\rangle$ be two general points of $VSP_{G}^{X}(h)$. By Proposition \ref{p1} the varieties $V_{\{x_{1},...,x_{\overline{h}}\},n-k}$ and $V_{\{y_{1},...,y_{\overline{h}}\},n-k}$ are rationally connected. Furthermore we have
$$\alpha(V_{\{x_{1},...,x_{\overline{h}}\},n-k})\subseteq\overline{\beta^{-1}(H_{x})}$$
$$\alpha(V_{\{y_{1},...,y_{\overline{h}}\},n-k})\subseteq\overline{\beta^{-1}(H_{y})}$$
and
$$\alpha(V_{\{x_{1},...,x_{\overline{h}},y_{1},...,y_{\overline{h}}\},n-k})\subseteq\overline{\beta^{-1}(H_{x})}\cap \overline{\beta^{-1}(H_{y})}.$$
Therefore $V_{p}$ is rationally chain connected by chains of two rational curves intersecting in a general point of $\alpha(V_{\{y_{1},...,y_{\overline{h}}\},n-k})$.\\
Now, let $M_{p}\subset V_{p}$ be the variety parametrizing $\overline{h}$-secant varieties having more than one $\overline{h}$-secant $(\overline{h}-1)$-plane through $p$. By Lemma \ref{l1} a general $\sec_{\overline{h}}(Y)\in V$ is singular in codimension one. Therefore $M_{p}$ has codimension two in $V$ and since $M_{p}\subset V_{p}\subset V$ we have
$$\codim_{V_{p}}M_{p} = 1.$$
Assume that the general point of $\Sing(\sec_{\overline{h}}(Y))$ is of multiplicity $t\geq 2$. By Terracini's lemma \cite[Theorem 1.1]{CC} there are $t$ $\overline{h}$-secant $(\overline{h}-1)$-planes trough a general point of $\Sing(\sec_{\overline{h}}(Y))$. In particular there exist two zero dimensional subschemes $\{z_{1},...,z_{\overline{h}}\}$ and $\{w_{1},...,w_{\overline{h}}\}$ such that 
$$\sec_{\overline{h}}(Y)\in\alpha(V_{\{z_{1},...,z_{\overline{h}},w_{1},...,w_{\overline{h}}\},n-k}).$$
Therefore $V_{p}$ is rationally chain connected by chains of rational curves intersecting in general points of $M_{p}$. Let us consider the normalization 
$$\nu:\widetilde{V}_{p}\rightarrow V_{p}$$
and the intersection
$$\mathcal{I}_{x,y}=\alpha(V_{\{x_{1},...,x_{\overline{h}}\},n-k})\cap\alpha(V_{\{y_{1},...,y_{\overline{h}}\},n-k}).$$
Then $\dim(\mathcal{I}_{x,y})\geq a$. Let $\widetilde{M}_{p} = \nu^{-1}(M_{p})$. We have that $\nu_{|\widetilde{M}_{p}}$ is a finite and $\rm\acute{e}$tale morphism outside of a codimension one set $K$. For any point $q\in M_{p}\setminus K$ there is an $\rm\acute{e}$tale open neighborhood $U_{q}$ such that $\nu_{|\widetilde{M}_{p}}$ restricted to $\nu_{|\widetilde{M}_{p}}^{-1}(U_{q})$ is finite and $\rm\acute{e}$tale. Now, in the $\rm\acute{e}$tale topology there exists an open subset $B\subset M_{p}$ such that $K\subset B$. The complement $B^{c}$ is compact. Then we can cover $B^{c}$ by finitely many $\{U_{q_{i}}\}_{i=1,...,r}$. Now, the map $\nu_{|\widetilde{M}_{p}}$ restricted to $\nu_{|\widetilde{M}_{p}}^{-1}(U_{q_{i}})$ is $\rm\acute{e}$tale. Since $\{x_{1},...,x_{\overline{h}}\}$ and $\{y_{1},...,y_{\overline{h}}\}$ are general, $V$ is irreducible and the $\{U_{q_{i}}\}_{i=1,...,r}$ are finitely many we have
$$\dim(\nu_{|\widetilde{M}_{p}}^{-1}(V_{\{x_{1},...,x_{\overline{h}}\},n-k})\cap \nu_{|\widetilde{M}_{p}}^{-1}(V_{\{y_{1},...,y_{\overline{h}}\},n-k}))> 0.$$
Therefore $\widetilde{V}_{p}$ is rationally chain connected by chains of rational curves though general points of $\nu^{-1}(M_{p})$. Then $\widetilde{V}_{p}$ and $V_{p}$ are rationally connected as well. Finally, since the rational map $\beta$ is dominant $VSP^{X}_{H}(\overline{h})$ is rationally connected.
\end{proof}

\begin{Remark}\label{sharp}
When $X = V_{d}^{n}$ is the Veronese variety, from Theorem \ref{RC2} we recover \cite[Theorem 4.1]{MMe}. In \cite{IR} \textit{A. Iliev} and \textit{K. Ranestad} proved that if $X = V_{3}^{5}$ then $VSP_{H}^{X}(10)$ is a Hyperk\"ahler manifold deformation equivalent to the Hilbert square of a $K3$ surface of genus $8$. In particular $VSP_{H}^{X}(10)$ can not be rationally connected. In this case we have $N(n,d) = \bin{n+d}{n}-1 = 55$, so $k+1 = 5$, and Theorem \ref{RC2} holds for $h\geq 11$.  
\end{Remark}

In what follows we work out some numbers which make Theorem \ref{RC2} work. We denote by $\overline{h}$ the smallest $h$ for which Theorem \ref{RC2} holds.

\subsubsection*{Grassmannians}
The Grassmannian $\G(r,n)$ parametrizing $r$-linear subspaces of $\mathbb{P}^{n}$ is a rational homogeneous variety of dimension $(r+1)(n-r)$, and has a natural embedding
$$\G(r,n)\hookrightarrow \mathbb{P}^{N},$$
with $N = \binom{n+1}{r+1}-1$, called the Pl\"ucker embedding.
\begin{center}
{\small
\begin{tabular}{|c|c|c|c|c|c|}
\hline
$r$ & $n$ & $\dim(\G(r,n))$ & $N$ & $k$ & $\overline{h}$\\
\hline
1 & 4 & 6 & 9 & 2 & 3\\
1 & 5 & 8 & 14 & 6 & 3\\
2 & 6 & 12 & 34 & 1 & 17\\ 
2 & 7 & 15 & 55 & 10 & 5\\
3 & 8 & 20 & 125 & 4 & 25\\
\hline
\end{tabular}
}
\end{center}
\subsubsection*{Segre-Veronese Varieties}
Combining the Segre and the Veronese embeddings we can define the Segre-Veronese embedding
\begin{center}
$\psi: \mathbb{P}^{n}\times \mathbb{P}^{m} \rightarrow \mathbb{P}^{N}$,
\end{center}
with $N = \binom{a+n}{n}\binom{b+m}{m}-1$, using the sheaf $\mathcal{O}_{\mathbb{P}^{n}}(a)$ on $\mathbb{P}^{n}$ and the sheaf $\mathcal{O}_{\mathbb{P}^{m}}(b)$ on $\mathbb{P}^{m}$. Let $SV_{a,b}^{n,m} = \psi(\mathbb{P}^{n}\times \mathbb{P}^{m})$ be the Segre-Veronese variety.
\begin{center}
{\small
\begin{tabular}{|c|c|c|c|c|c|c|c|}
\hline
$n$ & $m$ & $a$ & $b$ & $\dim(SV_{a,b}^{n,m})$ & $N$ & $k$ & $\overline{h}$\\
\hline
2 & 3 & 1 & 3 & 5 & 39 & 2 & 13\\
4 & 4 & 2 & 3 & 8 & 524 & 3 & 131\\
4 &	4 &	3 &	3 &	8 &	1224 & 3 & 153\\
5 &	5 &	3 &	3 &	10 & 3135 & 4 & 627\\
5 &	5 &	3 &	4 &	10 & 7055 & 4 & 1411\\
\hline
\end{tabular}
}
\end{center}

\section{Canonical decompositions and unirationality}\label{candec}
In this section we consider the case when there exists a positive integer $\overline{h}$ such that $\sec_{\overline{h}}(X) = \mathbb{P}^{N}$ and $VSP_{H}^{X}(\overline{h})$ is a single point, that is through a general point $p\in\sec_{\overline{h}}(X)$ passes exactly one $(\overline{h}-1)$-plane $\overline{h}$-secant to $X$. In the following we prove that the existence of such a canonical decomposition yields the unirationality of $VSP_{H}^{X}(h)$ for $h\geq\overline{h}$.

\begin{Theorem}\label{uu}
Let $X\subset\mathbb{P}^{N}$ be a rational variety. Assume that there exists a positive integer $\overline{h}$ such that $\sec_{\overline{h}}(X) = \mathbb{P}^{N}$ and $VSP_{H}^{X}(\overline{h})$ is a single point. Then $VSP_{H}^{X}(h)$ is unirational for any $h\geq\overline{h}$.
\end{Theorem}
\begin{proof}
As usual, let us fix a general point $p\in\mathbb{P}^{N}$. For any $h>\overline{h}$ let us consider the incidence variety
   \[
  \begin{tikzpicture}[xscale=1.5,yscale=-1.2]
    \node (A0_1) at (1, 0) {$\mathcal{I} = \{\{x_{1},...,x_{h-\overline{h}},q\}\; |\; q\in \left\langle p,x_{1},...,x_{h-\overline{h}}\right\rangle\}\subseteq(X)^{h-\overline{h}}\times\mathbb{P}^{N}$};
    \node (A1_0) at (0, 1) {$(X)^{h-\overline{h}}$};
    \node (A1_2) at (2, 1) {$\mathbb{P}^{N}$};
    \path (A0_1) edge [->]node [auto] {$\scriptstyle{\psi}$} (A1_2);
    \path (A0_1) edge [->]node [auto,swap] {$\scriptstyle{\phi}$} (A1_0);
  \end{tikzpicture}
  \]
Now, $X$ is rational, $\phi$ is dominant and its general fiber is a linear subspace of dimension $h-\overline{h}$ of $\mathbb{P}^{N}$. Therefore $\mathcal{I}$ is rational as well. By hypothesis there exists a unique zero dimensional subscheme $\{y_{1},...,y_{\overline{h}}\}$ spanning a $(\overline{h}-1)$-plane containing $q$. Furthermore $q\in\left\langle x_{1},...,x_{h-\overline{h}},p\right\rangle$ and $q\in\left\langle y_{1},...,y_{\overline{h}}\right\rangle$ imply 
$$p\in\left\langle y_{1},...,y_{\overline{h}},x_{1},...,x_{h-\overline{h}}\right\rangle.$$
Therefore we have the generically finite rational map
$$
\begin{array}{cccc}
\chi: & \mathcal{I} & \dasharrow & VSP_{H}^{X}(h)\\
 & \{x_{1},...,x_{h-\overline{h}}\} & \longmapsto & \{y_{1},...,y_{\overline{h}},x_{1},...,x_{h-\overline{h}}\}
\end{array}
$$
Under our hypothesis on $\sec_{\overline{h}}(X)$ we have
$$\dim(\sec_{\overline{h}}(X)) = n\overline{h}+\overline{h}-1 = N,$$
where $n = \dim(X)$. Furthermore $\dim(\mathcal{I}) = n(h-\overline{h})+h-\overline{h} = nh-h-(n\overline{h}+\overline{h})$. Now, substituting $n\overline{h}+\overline{h} = N+1$ we get $\dim(\mathcal{I}) = h(n+1)-N-1 = \dim(VSP_{H}^{X}(h))$. Therefore $\chi$ is dominant and $VSP_{H}^{X}(h)$ is unirational.
\end{proof}

\begin{Remark}
Under the hypothesis of Theorem \ref{uu} the natural map $\Sec_{\overline{h}}(X)\rightarrow\mathbb{P}^{N}$ is dominant and birational. Then, by \cite[Theorem 2.1]{Me2}, the variety $X$ is rational. 
\end{Remark}

In the following we consider the classical case $X = V_{d}^{n}$ of the Veronese variety. When the secant variety of the Veronese variety fills the projective space there are few cases in which we have the uniqueness of the decomposition.

\begin{Theorem}\cite[Theorem 1]{Me2}
Fix integers $d > n > 1$. Then a general homogeneous polynomial $F\in k[x_{0},...,x_{n}]_{d}$ can be expressed as a sum of $d$-th powers of linear forms in a unique way if and only if $d = 5$ and $n = 2$.
\end{Theorem} 

It is known that a general homogeneous polynomial $F\in k[x_{0},...,x_{n}]_{d}$ admits a canonical decomposition as a sum of $d$-th powers of linear forms in the following cases.
\begin{itemize}
\item[-] $n = 1$, $d = 2h-1$ \cite{Sy},
\item[-] $n = d = 3$, $h = 5$ \cite{Sy},
\item[-] $n = 2$, $d = 5$, $h = 7$ \cite{Hi}.
\end{itemize}

We give very simple and geometrical proofs of these facts.

\begin{Proposition}\label{sy1}
Let $F\in k[x_{0},x_{1}]_{2h-1}$ be a general homogeneous polynomial. There exists a unique decomposition of $F$ as sum of $h$ linear forms. 
\end{Proposition}
\begin{proof} Let $X$ be the rational normal curve of degree $2h-1$ in $\mathbb{P}^{2h-1}$. Since $\dim(\sec_{h}(X)) = h+(h-1) = 2h-1$ there exists a decomposition of $F$.\\
Suppose that $\{l_{1},...,l_{h}\}$ and $\{L_{1},...,L_{h}\}$ are two distinct decompositions of $F$. Let $H_{l}$ and $H_{L}$ be the two $(h-1)$-planes generated by the decompositions. The point $F_{2h-1}$ belongs to $H_{l} \cap H_{L}$. So the linear space $H = \langle H_{l},H_{L}\rangle$ has dimension 
$$
\dim(H) \leq (h-1)+(h-1) = 2h-2.
$$
If $H_{l} \cap H_{L} = \{F\}$, then $\dim(H) = (h-1)+(h-1) = 2h-2$. So $H$ is a hyperplane in $\mathbb{P}^{2h-1}$ and $H\cdot X \geq 2h$. A contradiction because $\deg(X) = 2h-1$.\\
If $H_{l}$ and $H_{L}$ have $s$ common points, then $H_{l}$ and $H_{L}$ intersect in $s+1$ points $q_{1},...,q_{s},F$. In this case $H_{l} \cap H_{L}\cong\mathbb{P}^{s}$ and $\dim(H) = 2h-2-s$. We choose $s$ points $p_{1},...,p_{s}$ on $X$ in general position. So $\Pi = \langle H,p_{1},...,p_{k}\rangle$ is a hyperplane such that $\Pi\cdot X \geq 2h-s+s = 2h$, a contradiction. We conclude that the decomposition of $F$ in $h$ linear factors is unique.   
\end{proof}

In order to reconstruct the decomposition we consider the following construction.
\begin{Construction}\label{cn1}
The partial derivatives of order $h-2$ of $F$ are $\binom{h-2+1}{1} = h-1$ homogeneous polynomials of degree $h+1$. Let $\nu_{h+1}^{1}:\mathbb{P}^{1}\rightarrow \mathbb{P}^{h+1}$ be the Veronese embedding and let $X = \nu_{h+1}^{1}(\mathbb{P}^{1})$ be the corresponding rational normal curve. Consider the projection
$$\pi:\mathbb{P}^{h+1}\dasharrow \mathbb{P}^{2}$$
from the $(h-2)$-plane $H_{\partial}$ spanned by the partial derivatives. Since the decomposition $\{L_{1},...,L_{h}\}$ of $F$ is unique, the projection $\overline{X} = \pi(X)$ has a unique singular point $P = \pi(\langle L_{1}^{h+1},...,L_{h}^{h+1}\rangle)$ of multiplicity $h$. Now to find the decomposition we have to compute the intersection $H \cdot X = \{L_{1}^{h+1},...,L_{h}^{h+1}\}$, where $H = \langle H_{\partial},P\rangle$. 
\end{Construction}

\begin{Example}
We consider the polynomial
$$F = x_0^{3} + x_0^{2}x_1 - x_0x_1^{2} + x_1^{3} \in k[x_0,x_1]_{3}.$$
that is the point $[1 : 1 : 1 : 1] \in \mathbb{P}^{3}$. The projection from $[1 : 1 : 1 : 1]$ to the plane
$\{X = 0\} \cong \mathbb{P}^{2}$ is given by
$$\pi:\mathbb{P}^{3}\dasharrow \mathbb{P}^{2}, \: [X : Y : Z : W ] \mapsto [Y-X : X+Z : W-X].$$  
We compute the projection $C = \pi(X)$
of the twisted cubic curve $X$, and the singular point of $C$,
$$P = \Sing(C) = [4 : 10 : 9].$$
The line $L = \langle [1 : 1 : 1 : 1],P\rangle$ is given by the following equations
$$
\left\{
\begin{array}{l}
3X-5Y-2Z = 0,\\
5-9Y+4W = 0.
\end{array}
\right.
$$
The intersection $X\cdot L$ is given by 
$$
\left.
\begin{array}{l}
L_{1}^{3} = [0.0515957 : 0.4157801 : 1.1168439 : 1],\\
L_{2}^{3} = [155.0515957 : 86.5842198 : 16.1168439 : 1].
\end{array}
\right.
$$
These points correspond to the linear forms
$$
\left.
\begin{array}{l}
L_{1} = -0.3722812x_0 + x_1,\\
L_{2} = 5.3722813x_0 + x_1.
\end{array}
\right.
$$
Indeed we have
$$F = 0.99322 \cdot (-0.3722812x_0 + x_1)^{3} + 0.00678 \cdot (5.3722813x_0 + x_1)^{3}.$$
\end{Example}
Now we consider Hilbert's theorem.

\begin{Theorem}\label{hi}
Let $F \in k[x_{0},x_{1},x_{2}]_{5}$ be a general homogeneous polynomial. Then $F$ can be decomposed as sum of seven linear forms
$$F = L_{1}^{5}+...+L_{7}^{5}.$$
Furthermore the decomposition is unique.
\end{Theorem}
\begin{proof}
By the main result in \cite{AH} we have $\dim(VSP(F,7)) = 0$. Assume that $F$ admits two different decompositions $\{L_{1},...,L_{7}\}$ and $\{l_{1},...,l_{7}\}$. Consider the second partial derivatives of $F$. These are six general homogeneous polynomials of degree three. Let $H_{\partial}\subseteq \mathbb{P}^{9}$ be the linear space they generate. Clearly a decomposition of $F$ induces a decomposition of its partial derivatives and we have 
$$H_{L} =\langle L_{1}^{3},...,L_{7}^{3}\rangle\supset H_\partial\subset \langle l_{1}^{3},...,l_{7}^{3}\rangle=H_l.$$ 
Since $F$ is general both $H_{L}$ and $H_{l}$ intersect the Veronese surface $V_{3}^{2} \subseteq \mathbb{P}^{9}$ at $7$ distinct points. Let
$$\pi : \mathbb{P}^{9}\dasharrow \mathbb{P}^{3}$$
be the projection from $H_\partial$, and $\overline{V} = \pi(V_{3}^{2})$. Then $\overline{V}$ is a surface  of degree $\deg(\overline{V}) = 9$ with two points of multiplicity $7$ corresponding to $\pi(H_L)$ and $\pi(H_l)$. The line $\left\langle\pi(H_L),\pi(H_l)\right\rangle$ intersects $\overline{V}$ with multiplicity at least $14>\deg(\overline{V})$. Then $\left\langle\pi(H_L),\pi(H_l)\right\rangle\subset\overline{V}$ and the $7$-plane $H:=\langle H_L,H_l\rangle$ intersect
$V_{3}^{2}$ along a curve $\Gamma$ corresponding to $\left\langle\pi(H_L),\pi(H_l)\right\rangle$. The construction of $\Gamma$ yields
$$\deg\Gamma\leq (H_L\cdot V_{3}^{2})=7.$$
On the other hand $\deg\Gamma=3j$ therefore we end up with the following possibilities.
\begin{itemize}
\item[-] If $\deg\Gamma=3$ then $\Gamma$ is a twisted cubic curve contained in $H$ and 
$$H_{l} \cdot \Gamma = H_{L} \cdot \Gamma = 3.$$
We may assume that $H_{l} \cap \Gamma = \{l^{3}_{1},l^{3}_{2},l^{3}_{3}\}$ and $H_{L} \cap \Gamma = \{L^{3}_{1},L^{3}_{2},L^{3}_{3}\}$. Let
$\Lambda$ be the pencil of hyperplanes containing $H$, and $\nu_{3}^{2}:\P^2\rightarrow V_{3}^{2}$ the Veronese embedding. The linear system
$\nu_{3}^{2*}(\Lambda_{|V_{3}^{2}})$ is a pencil of conics and therefore the base locus of $\Lambda_{|V_{3}^{2}}$ consists of at least four points. To conclude observe that the base locus of $\Lambda_{|V_{3}^{2}}$ contains $H\cap V_{3}^{2}$. This forces
$$\{L^{3}_{4},L^{3}_{5},L^{3}_{6},L^{3}_{7}\}=\{l^{3}_{4},l^{3}_{5},l^{3}_{6},l^{3}_{7}\},$$
and consequently $H_L=H_l$.
\item[-] If $\deg\Gamma=6$ then 
$$H_{l} \cdot \Gamma = H_{L}\cdot \Gamma = 6.$$
We may assume that $\Gamma\supset\{L_1^3,\ldots,L_6^3\}\cup\{l_1^3,\ldots,l_6^3\}$. Let $\Lambda$ be the pencil of hyperplanes containing $H$. The linear system $\nu_{3}^{2*}(\Lambda_{|V_{3}^{2}})$ is a pencil of lines and therefore the base locus of $\Lambda_{|V_{3}^{2}}$ consists of at least one point. This forces
$$L^{3}_{7}=l^{3}_{7},$$
and therefore $H_L=H_l$.
\end{itemize}
In both cases we have $H_L=H_l$ and so the contradiction $\{L_{1},...,L_{7}\}=\{l_{1},...,l_{7}\}$.
\end{proof}

The following construction is inspired by the proof of Theorem \ref{hi}, and provides a method to reconstruct the decomposition.
\begin{Construction}\label{cHi}
Let $\{L_{1},...,L_{7}\}$ be the unique decomposition of $F$. Consider now the projection
$$\pi:\mathbb{P}^{9}\dasharrow \mathbb{P}^3$$
from the linear space $H_{\partial}$. The image of the Veronese variety $\overline{V}$ is a surface of degree $9$ in $\mathbb{P}^{3}$. Furthermore it has a point $P$ of multiplicity $7$, which comes from the contraction of $H_{L} = \left\langle L_{1}^{3},...,L_{7}^{3}\right\rangle$. This is the unique point of multiplicity $7$ on $\overline{V}$ by the uniqueness of the decomposition. From this discussion we derive the following algorithm to find the decomposition.
\begin{itemize}
\item[-] Compute the second partial derivative of $F$.
\item[-] Compute the equation of the $5$-plane $H_{\partial}$ spanned by the derivatives.
\item[-] Project the Veronese variety $V$ in $\mathbb{P}^{3}$ from $H_{\partial}$.
\item[-] Compute the point $P\in\overline{V}$ of multiplicity $7$.
\item[-] Compute the $6$-plane $H = \langle H_{\partial},P\rangle$ spanned by $H_{\partial}$ and the point $P$.
\item[-] Compute the intersection $V_{3}^{2}\cdot H = \{L_{1}^{3},...,L_{7}^{3}\}$.
\end{itemize}  
\end{Construction}

We can prove the uniqueness theorem for $n = 3$, $d = 3$ and $h = 5$ with a variation on the argument we used in the proof of Theorem \ref{hi}. It is enough to consider the first partial derivatives of a general polynomial $F\in k[x_0,x_1,x_2,x_3]_{3}$ instead of the second partial derivatives. However \textit{G. Ottaviani} pointed out the following very simple proof of Sylvester's pentahedral theorem.

\begin{Theorem}\label{sy}
Let $F\in k[x_0,x_1,x_2,x_3]_{3}$ be a general homogeneous polynomial. Then $F$ can be decomposed as sum of five linear forms
$$F = L_{1}^{3}+...+L_{5}^{3}.$$
Furthermore the decomposition is unique.
\end{Theorem}
\begin{proof}
By the main result in \cite{AH} we have $\dim(VSP(F,5)) = 0$. Let us consider
$$P_{\xi}F = \xi_{0}\frac{\partial F}{\partial x_{0}} + \xi_{1}\frac{\partial F}{\partial x_{1}} + \xi_{2}\frac{\partial F}{\partial x_{2}} + \xi_{3}\frac{\partial F}{\partial x_{3}}.$$
If $F = L^{3}_{1}+...+L^{3}_{5}$ then we have
$$P_{\xi}F = \sum^{5}_{i = 1}\xi_{i}\lambda_{i}L^{2}_{i}.$$
Therefore $P_{\xi}F$ has rank $2$ on the points $\xi \in \mathbb{P}^{3}$ on which three of the linear forms $L_{i}$ vanish simultaneously. These points are $\binom{5}{3} = 10$.\\
Now we consider the subvariety $X_{2}$ of $\mathbb{P}^{9}$ parametrizing rank $2$ quadrics. A quadric $Q$ of rank $2$ is the union of two planes. We have $\dim(X_{2}) = 6$. To find the degree of $X_{2}$ we have to intersect it with a $3$-plane. So the degree of $X_{2}$ is equal to the number of quadrics of rank $2$ passing through $6$ general points of $\mathbb{P}^{3}$. These quadrics are $\frac{1}{2}\binom{6}{3} = 10$. Therefore $\dim(X_{2}) = 6$ and $\deg(X_{2}) = 10$. Now, let us consider the $3$-plane
$$\Gamma = \{P_{\xi}F \: | \: \xi \in \mathbb{P}^{3}\} \subseteq \mathbb{P}^{9}.$$   
The intersection $\Gamma\cap X_{2} = \{P_{\xi}F \: | \: \rank(P_{\xi}F) = 2\}$ is a set of $10$ points. These points have to be the $10$ points we have found in the first part of the proof. Then the decomposition of $F$ is unique.
\end{proof}

The argument used in the proof suggests us an algorithm to reconstruct the decomposition.
\begin{Construction}
Consider $F$ and its first partial derivatives.
\begin{itemize}
\item[-] Compute the $3$-plane $\Gamma$ spanned by the partial derivatives of $F$.
\item[-] Compute the intersection $\Gamma \cdot X_{2}$, where $X_{2}$ is the variety parametrizing the rank $2$ quadrics in $\mathbb{P}^{3}$.
\item[-] Consider the $10$ points in the intersection. On each plane we are looking for there are $6$ of these points, furthermore on each plane there are $4$ triples of collinear points. Then with these $10$ points we can construct exactly $\frac{\binom{10}{3}}{\binom{6}{3}+4} = 5$ planes. These planes give the decomposition of $F$. Note that a priori we have $\binom{10}{6} = 210$ choices, but we are interested in sets of six points $\{P_{j_{1}},...,P_{j_{6}}\}$ which lie on the same plane. We know that there are exactly five of these. To find the five combinations we use the following Mathlab script which constructs a matrix $A$ whose lines are the ten points and then computes the $6\times 4$ submatrices of rank $3$ of $A$.
\begin{scr}
$
\texttt{P1 = input('Point 1:');}\\
\vdots\\
\texttt{P10 = input('Point 10:');}\\
\texttt{q = input('Precision:');}\\
\texttt{A = [P1;P2;P3;P4;P5;P6;P7;P8;P9;P10];}\\
\texttt{t = 1;}\\
\texttt{B = [];}\\
\texttt{for a=1:5,}\\
\texttt{for b=a+1:6,}\\
\texttt{for c=b+1:7,}\\
\texttt{for d=c+1:8,}\\
\texttt{for f=d+1:9,}\\
\texttt{for g=f+1:10,}\\
\texttt{M = [A(a,:);A(b,:);A(c,:);A(d,:);A(f,:);A(g,:)];}\\
\texttt{disp(t);}\\
\texttt{t = t+1;}\\
\texttt{v = [];}\\
\texttt{for a1 = 1:3,}\\
\texttt{for a2 = a1+1:4,}\\
\texttt{for a3 = a2+1:5,}\\
\texttt{for a4 = a3+1:6,}\\
\texttt{v = [v,det([M(a1,:);M(a2,:);M(a3,:);M(a4,:)])];}\\
\texttt{end; end; end; end;}\\
\texttt{if abs(v(1))<q,abs(v(2))<q,abs(v(3))<q,abs(v(4))<q,abs(v(5))<q,}\\
\texttt{abs(v(6))<q,abs(v(7))<q,abs(v(8))<q,abs(v(9))<q,abs(v(10))<q,}\\
\texttt{abs(v(11))<q,abs(v(12))<q,abs(v(13))<q,abs(v(14))<q,abs(v(15))<q,}\\	
\texttt{B = [B M];}\\
\texttt{end; end; end; end; end; end; end;}\\
\texttt{[n,m] = size(B);}\\
\texttt{s = 1;}\\
\texttt{for r=1:4:m-3,}\\
\texttt{disp('Matrix'), disp(s),}\\
\texttt{s = s+1;}\\
\texttt{B(:,r:r+3),}\\
\texttt{end;}
$
\end{scr}   
\end{itemize}
\end{Construction}

In the following we work out Theorem \ref{uu} in the case $X = V_{5}^{2}$ and $h\geq 7$.

\begin{Proposition}\label{uni}
Let $F\in k[x_{0},x_{1},x_{2}]_{5}$ be a general homogeneous polynomial. For any $h\geq 7$ the variety $VSP(F,h)$ is unirational.
\end{Proposition}
\begin{proof}
By Theorem \ref{hi} $VSP(F,7)$ is a single point. If $h\geq 8$ consider the incidence variety 
    \[
  \begin{tikzpicture}[xscale=1.5,yscale=-1.2]
    \node (A0_1) at (1, 0) {$\mathcal{I} = \{(l_{1},...,l_{h-7},G)\; |\; G\in \langle F,l_{1}^{5},...,l_{h-7}^{5}\rangle\}\subseteq(\mathbb{P}^{2})^{h-7}\times\mathbb{P}^{20}$};
    \node (A1_0) at (0, 1) {$(\mathbb{P}^{2})^{h-7}$};
    \node (A1_2) at (2, 1) {$\mathbb{P}^{20}$};
    \path (A0_1) edge [->]node [auto] {$\scriptstyle{\psi}$} (A1_2);
    \path (A0_1) edge [->]node [auto,swap] {$\scriptstyle{\phi}$} (A1_0);
  \end{tikzpicture}
  \]
The map $\phi$ is dominant and its general fiber is a linear subspace of dimension $h-7$ in $\mathbb{P}^{20}$. Then $\mathcal{I}$ is a rational variety of dimension $2(h-7)+h-7 = 3h-21$.\\
Let $(l_{1},...,l_{h-7},G)\in\mathcal{I}$ be a general point. By Theorem \ref{hi} the polynomial $G$ admits a unique decomposition $G = L_{1}^{5}+...+L_{7}^{5}$. Since $G\in \langle F,l_{1}^{5},...,l_{h-7}^{5}\rangle$ we have $L_{1}^{5}+...+L_{7}^{5} = \alpha F + \sum_{i=1}^{h-7}\lambda_{i}l_{i}^{5}$, and
$$F = \frac{1}{\alpha}L_{1}^{5}+...+\frac{1}{\alpha}L_{7}^{5}-\sum_{i=1}^{h-7}\frac{\lambda_{i}}{\alpha}l_{i}^{5}.$$
We get a generically finite rational map
$$
\begin{array}{cccc}
\chi: & \mathcal{I} & \dasharrow & VSP(F,h)\\
 & (l_{1},...,l_{h-7},G) & \longmapsto & \{L_{1},...,L_{7},l_{1},...,l_{h-7}\}
\end{array}
$$
Since $\dim(VSP(F,h)) = 3h-21 = \dim(\mathcal{I})$ the map $\chi$ is dominant and $VSP(F,h)$ is unirational.
\end{proof}

\begin{Remark}
Consider a general homogeneous polynomial $F\in k[x_0,x_1,x_2,x_3]_{3}$. By Theorem \ref{sy} $F$ admits a unique decomposition as sum of five powers of linear forms. The argument used in Proposition \ref{uni} in this case shows that $VSP(F,h)$ is unirational for any $h\geq 5$.
\end{Remark}

\subsubsection*{Acknowledgments}
I would like to thank \textit{Massimiliano Mella} for many helpful comments.

\end{document}